\documentclass{article}
\usepackage[english]{babel}
\usepackage[T1]{fontenc}
\usepackage[utf8]{inputenc}
\usepackage{amsthm}
\usepackage{enumerate}
\usepackage{amsfonts}
\usepackage{verbatim}
\usepackage{graphicx}
\usepackage[babel=true]{csquotes}
\usepackage{amsmath}
\usepackage{amssymb}
\usepackage{mathrsfs}
\usepackage{mathtools}
\newtheorem{defi}{Definition}
\newtheorem{theo}{Theorem}
\newtheorem{prop}{Proposition}
\newtheorem{conj}{Conjecture}
\newtheorem{lem}{Lemma}
\newtheorem{coro}{Corollary}
\title{On a new conjecture about super-monochromatic factorisations and ultimate periodicity}
\date{Ca\"{\i}us Wojcik}

\begin{document}

\maketitle

\begin{abstract}
We study a conjecture linking ultimate periodicity of infinite words to the existence of colorings on finite words avoiding monochromatic factorisation of suffixes, with the extra condition that the ordered concatenation of elements of this factorisation remains monochromatic. This type of results shows the limits of Ramsey theory in the context of combinatorics on words. We show some reductions of the problem and the example of the Zimin word. Using the new notion of consecutive length, we show that squarefree words fulfill the conjecture.

\end{abstract}

\bigskip

In \cite{wojcik}, we showed the following theorem :

\begin{theo}
Let $x$ be an infinite word over an alphabet $\mathcal{A}$. Then $x$ is periodic if and only if for all coloring of its set of factors, $x$ admits a monochromatic factorisation.
\end{theo}

The proof of this theorem consists in building a particular coloring, depending on a non-periodic word $x$, such that $x$ admits no factorisation (into finite words) with all factors of the same color.

We introduce and study now a more recent conjecture, stated as follows :

\begin{conj}
Let $x$ be an infinite word over an alphabet $\mathcal{A}$. Then $x$ is ultimately-periodic if and only if for all coloring of the set of finite words over $\mathcal{A}$, $x$ admits a suffix having a super-monochromatic factorisation.
\end{conj}

A super-monochromatic factorisation is an expression of the form $u_1u_2u_3\ldots$ where the $(u_i)$'s are finite words and the set of words $(u_{n_1}u_{n_2}\ldots u_{n_k})$ for $k\geq 1$ and $n_1< n_2< \ldots < n_k$ is monochromatic. As for the previously stated theorem, our work will consist in building colorings forbidding the monochromaticity of certain substructures.

\bigskip

This paper is organized as follows. In a short introduction we present some notations and basic definitions from combinatorics on words. In a first part we present this conjecture in the context of Ramsey theory. Namely we use a theorem of Hindman to show that an infinite word $x$ admits an element in its subshift having a super-monochromatic factorisation. In a second part we prove some reductions of the problem, in particular that we may freely assume that the words $(u_{n_1}u_{n_2}\ldots u_{n_k})$ are factors of our base word $x$. In a third part we study the Zimin word and provide an optimal construction for this word and the related period-doubling word. In a fourth part we introduce the consecutive length associated to $x$ and show some of its properties. Finally, we build a coloring showing that squarefree words satisfy the conjecture, with three colors.

\section{Introduction}

In this part we present some basic definitions used throughout the paper.

\bigskip

Let $\mathcal{A}$ be a set, called the alphabet, which may be finite or infinite. We will be working with a fixed infinite word $x$, so that $\mathcal{A}$ will be taken countable, equal to the set of letters appearing in $x$.

An infinite word $x$ is an element of $\mathcal{A}^\mathbb{N}$. A finite word is an element of $\cup_{n\geq 1} \mathcal{A}^n$, and we note $\mathcal{A}^+$ the set of finite words. If $u=a_1a_2\ldots a_n$ then we set $n=|u|$ its length and $\{a_1a_2\ldots a_i \ | \ 1\leq i \leq n\}$ its set of prefixes. If $x$ is a finite or infinite word and $n\geq 1$, we note $\mathbb{P}_n(x)$ its prefix of length $n$. A factor $u$ of $v$ is a finite word appearing in $v$.

Let $T: x_0x_1x_2\ldots \in \mathcal{A}^\mathbb{N} \longrightarrow x_1x_2x_3\ldots\in \mathcal{A}^\mathbb{N}$ denote the shift on infinite words, depriving an infinite word of its first letter. If $x$ is an infinite word, a suffix of $x$ is an element of the form $T^k(x)$ for some $k\geq 0$.

A factorisation of an infinite word $x$ is an expression of the form $x=u_1u_2u_3u_4\ldots$ where the $(u_i)'s$ are finite words.

\bigskip

Let $X$ be any set. A coloring on $X$ is any map : $C: X \longrightarrow C$ where $C$ is a finite set, making the abuse of noting the map $C$ and its image by the same symbol (which will always be $C$). We will call the set $C$ the set of colors, for example $C=\{\textbf{red}, \textbf{blue}\}$.

Suppose that a set $X$ and a coloring $C$ on $X$ are given. A subset $Y$ of $X$ is called monochromatic if the restriction of $C$ to $Y$ is constant.

\begin{defi}[Super-monochromatic factorisation]
Let $y$ be an infinite word over $\mathcal{A}$, and $C$ a coloring on $\mathcal{A}^+$. A factorisation
\begin{center}
	$y=u_1u_2u_3u_4\ldots$
\end{center}
is said to be super-monochromatic if the set of finite words :
\begin{center}
	$\displaystyle \left( u_{n_1}u_{n_2}\ldots u_{n_k} \right)_{k\geq 1, n_1<n_2<\ldots < n_k}$
\end{center}
is monochromatic.
\end{defi}

\section{Context of Ramsey theory}

In this section we put conjecture 1 in the context of Ramsey theory. We state Hindman's theorem, show the easy part of the conjecture, and show that for any coloring of $\mathcal{A}^+$ and infinite word $x$ there is $y$ in the subshift of $x$ having a super-monochromatic factorisation.

\bigskip

Denote by $\mathcal{P}(\mathbb{N})$ the set of finite subsets of $\mathbb{N}$. Write $A<B$ for $A,B\in \mathcal{P}(\mathbb{N})$ when $\max A < \min B$.

\begin{theo}[Hindman]
\begin{enumerate}[i)]
	\item For any coloring on $\mathbb{N}$, there exists an infinite subset $M=\{m_1<m_2<m_3<\ldots\}$ of $\mathbb{N}$ such that the set
	\begin{center}
		$\displaystyle \left( m_{n_1}+m_{n_2}+\ldots + m_{n_k} \right)_{k\geq 1, n_1<n_2<\ldots < n_k}$
	\end{center}
	is monochromatic.
	\item For any coloring on $\mathcal{P}(\mathbb{N})$, there exists an infinite subset $M=\{A_1<A_2<A_3<\ldots\}$ of $\mathcal{P}(\mathbb{N})$ such that the set
	\begin{center}
		$\displaystyle \left( A_{n_1}\cup A_{n_2}\cup\ldots \cup A_{n_k} \right)_{k\geq 1, n_1<n_2<\ldots < n_k}$
	\end{center}
	is monochromatic.
\end{enumerate}
\end{theo}

Hindman's theorem is known as part $i)$, and the implication $i)\Rightarrow ii)$ may be found easily in the litterature.

\bigskip

\begin{prop}
Let $x$ be an infinite word over $\mathcal{A}$, and $C$ a coloring of $\mathcal{A}^+$. If $x$ is ultimately-periodic, then $x$ admits a suffix having a super-monochromatic factorisation.
\end{prop}

\begin{proof}
Write $x=uvvv\ldots$ where $u$ and $v$ are finite words. Consider the coloring $\widetilde{C}$ of $\mathbb{N}^*$ defined for $i\geq 1$ by :
\begin{center}
	$\widetilde{C}(i)=C(v^i)$.
\end{center}
By Hindman's theorem, there exists an infinite subset $M=\{m_1<m_2<m_3<\ldots\}$ such that the set of words
\begin{center}
	$\displaystyle \left( v^{m_{n_1}+m_{n_2}+\ldots + m_{n_k}} \right)_{k\geq 1, n_1<n_2<\ldots < n_k}$
\end{center}
is monochromatic. So that the suffix $v^{m_1}v^{m_2}v^{m_3}\ldots$ has a super-monochromatic factorisation.
\end{proof}

\bigskip

Endow $\mathcal{A}$ with the discreet topology, and note
\begin{center}
	$\Omega (x) = \overline{\{T^{k}(x) \ | \ k\geq 0 \}}$
\end{center}
the subshift of $x$. Recall that a factor of $x$ is said to be recurrent if it appears an infinite number of times in $x$. The word $x$ itself is said to be recurrent if every factor of $x$ is recurrent. If $y$ is an infinite recurrent word such that every factor of $y$ is a factor of $x$, then $y\in \Omega (x)$. If $\mathcal{A}$ is finite, then $\Omega(x)$ admits a recurrent word by a compacity argument.

\begin{prop}
Let $x$ be an infinite word over $\mathcal{A}$ such that $\Omega(x)$ admits a recurrent element, and $C$ a coloring of $\mathcal{A}^+$. Then there exists $y\in \Omega (x)$ such that $y$ has a super-monochromatic factorisation.
\end{prop}

\begin{proof}
Let $z\in \Omega (x)$ be a recurrent word. We build a sequence $(u_n)_{n\geq 1}$ of factors of $z$ such that :
\begin{center}
	$\forall n\geq 1$ $u_1u_2\ldots u_n$ is a suffix of $u_{n+1}$,
\end{center}
a property that we call the suffix property, as follows. Take for $u_1$ any factor of $z$. Since $u_1$ is recurrent, there exists a finite word $v$ such that $u_1vu_1$ is a factor of $z$, and define $u_2=vu_1$. If $u_1$, $u_2$, \ldots $u_n$ are defined such that $\forall k=1\ldots n$ $u_1u_2\ldots u_{k-1}$ is a suffix of $u_{k}$ and $u_1u_2\ldots u_n$ is a factor of $z$, then there exists a finite word $v$ such that $u_1u_2\ldots u_n v  u_1u_2\ldots u_n$ is a factor of $z$. In this case, we set $u_{n+1}=v  u_1u_2\ldots u_n$. It is clear that $(u_n)$ defined by this induction process satisfy the desired property.

For a finite subset $A$ of $\mathbb{N}^*$, set
\begin{center}
	$\displaystyle u_A = \prod_{i\in A} u_i$
\end{center}
and apply Hindman's theorem to the coloring $\widetilde{C}$ of $\mathcal{P}(\mathbb{N})$ for $A\subset \mathbb{N}^*$ a finite set by :
\begin{center}
	$\widetilde{C}(A)=C(u_A)$
\end{center}
to obtain an infinite sequence $M=\{A_1<A_2< A_3 < \ldots\}$ such that the set of words
\begin{center}
	$\displaystyle \left(u_{A_{n_1}\cup A_{n_2}\cup\ldots \cup A_{n_k}}\right)_{k\geq 1, n_1<n_2<\ldots < n_k}$
\end{center}
is monochromatic. The suffix property implies that these words are factors of $x$. And since $u_{A\cup B}=u_Au_B$ whenever $A<B$, the set of concatenations 
\begin{center}
	$\displaystyle \left(u_{A_{n_1}}u_{A_{n_2}}\ldots u_{A_{n_k}}\right)_{k\geq 1, n_1<n_2<\ldots < n_k}$
\end{center}
is monochromatic and is a subset of the set of factors of $x$. This shows that the infinite word
\begin{center}
	$y=\displaystyle \prod_{n\geq 1}u_{A_n}$
\end{center}
belongs to $\Omega(x) $ and has a super-monochromatic factorisation.
\end{proof}

\section{A few reductions}

Let $x$ be an infinite word over $\mathcal{A}$. Let
\begin{center}
	$T^k(x)= u_1u_2u_3u_4\ldots$ \quad $(k\geq 0)$
\end{center}
be a factorisation of a suffix of $x$.

In this section we define some colorings that allow us to reduce the problem to a certain extend. We present three reductions. The first one allows us to reduce to the case where $x$ is recurrent at the cost of one color. The second allows us to reduce to the case where all the $(u_{n_1}u_{n_2}\ldots u_{n_k})_{k\geq 1, n_1<n_2<\ldots < n_k}$ are factors of $x$, at zero cost. The third one allows us to reduce to the case where the $(u_i)$ satisfy the suffix property, namely that $\forall n\geq 1$, $u_1u_2\ldots u_n$ is a suffix of $u_{n+1}$, at the cost of one color.

In this section, every statement concerning monochromatic factorisations is made with respect to the latest coloring defined. We start by the consideration of first occurrences of factors of $x$.

\bigskip

Write $x=x_0x_1x_2x_3\ldots$ where the $(x_i)$'s are elements of $\mathcal{A}$. For $u$ a factor of $x$, let :
\begin{center}
	$x_{A(u)}x_{A(u)+1}\ldots x_{B(u)-1}$
\end{center}
be the first occurrence of $u$ in $x$. In other words :
\begin{center}
	$A(u)=\min \{ k\geq 0 \ | \ u=x_kx_{k+1}\ldots x_{k+|u|-1} \}=\min \{ k\geq 0 \ | \ u=\mathbb{P}_{|u|}(T^k(x)) \}$
\end{center}
where we recall that $\mathbb{P}_{n}(y)$ is the prefix of length $n$ of $y$.

We have the obvious relation $B(u)-A(u)=|u|$. Moreover, if $v$ is a prefix of $u$ then $A(v)\leq A(u)$, and if $v$ is a suffix of $u$ then $B(v)\leq B(u)$. In this two functions resides the relative information between $u$ and $x$, and are relevant mostly in the non-ultimately periodic case, as shows the next proposition.

\begin{prop}
Let $x$ be a non-ultimately periodic word over $\mathcal{A}$. Then for all $k\geq 0$, there exist $N\geq 0$ such that for all $n\geq N$, we have
\begin{center}
	$k=A(\mathbb{P}_n(T^k(x)))$
\end{center}
in other word, for any suffix $y$ of $x$, a prefix of $y$ that is long enough has its first occurrence where $y$ starts.
\end{prop}

\begin{proof}
The infinite word $x$ is non-ultimately periodic if and only if
\begin{center}
	$\forall n,m \in \mathbb{N}$, \quad $n\neq m$ $\ \Longrightarrow \ \ T^n(x)\neq T^m(x) $.
\end{center}
For $k\geq 0$, we have $T^k(x) \neq T^j(x)$ for all $j=0\ldots k-1$, so that
\begin{center}
	$\forall j=0\ldots k-1$, $\exists N_j\geq 0$, $\forall n \geq N_j$, $\mathbb{P}_n(T^k(x))\neq \mathbb{P}_n(T^j(x))$.
\end{center}
This shows that for all $n\geq \max\{ N_j \ |\ j=0\ldots k-1 \}$, $\mathbb{P}_n(T^k(x))$ does not appear at a place $j$ for $j<k$. Hence the proposition.

\end{proof}

Thus in an expression of the form $T^k(x)=u_1u_2u_3\ldots$, if the $(u_i)$'s are long enough, we may assume that $A(u_1)=k$, and $B(u_i)=A(u_{i+1})$ for $i\geq 1$.

\bigskip

\begin{defi}
Let $C$ be the coloring defined on $\mathcal{A}^+$ for $u\in \mathcal{A}^+$ by :
\begin{itemize}
	\item $C(u)=\textbf{red}$ if $u$ is a factor of $x$ that is not recurrent,
	\item $C(u)=\textbf{blue}$ if $u$ is not a factor of $x$ or if $u$ is a recurrent factor of $x$.
\end{itemize}
\end{defi}

\begin{prop}
Assume that no suffix of $x$ is recurrent. Then no suffix of $x$ has a super-monochromatic factorisation for the coloring $C$ defined above.
\end{prop}

\begin{proof}
Assume by contradiction that $x$ has a suffix $T^k(x)= u_1u_2u_3u_4\ldots$ \quad $(k\geq 0)$ having a monochromatic factorisation. In view of the definition of a super-monochromatic factorisation, $u_1$ may be taken arbitrary long. Since $T^k(x)$ is not recurrent, if $u_1$ is long enough then it will contain a non-recurrent factor of $x$, and hence $u_1$ is $\textbf{red}$. This shows that the factorisation is $\textbf{red}$.

This implies that for all $n\geq 2$, the words $u_1u_n$ are $\textbf{red}$. In particular, they are factors of $x$. But since $x$ is non-ultimately periodic (if this was the case, $x$ would have a recurrent suffix), $\lim_n A(u_1u_n)= +\infty $ and $u_1$ appears an infinite number of times in $x$, contradicting the fact that it is non-recurrent.

\end{proof}

\begin{defi}
Let $C_{NF}$ be the coloring defined on the set of non-factors of $x$ for $u\in\mathcal{A}^+$ by :
\begin{itemize}
	\item $C_{NF}(u)=\textbf{red}$ if $u$ is not a factor of $x$ and every decomposition of $u$ as a product of factors of $x$ has at least tree terms
	\item $C_{NF}(u)=\textbf{blue}$ if $u$ is not a factor of $x$ and may be written as the concatenation of two factors of $x$
\end{itemize}
\end{defi}

Notice, in this definition, that we do not specify the color of factors of $x$.

\begin{prop}
Let $T^k(x)=u_1u_2u_3\ldots$ be a factorisation of a suffix of $x$, and for $A\subset \mathbb{N}^*$ a finite subset, set $u_A=\prod_{i\in A}u_i$. Assume that
\begin{center}
	$\forall \nu \geq 0, \exists A \in \mathcal{P}(\mathbb{N}^*)$ \ such that \ $\min A \geq \nu$ and $u_A$ is not a factor of $x$.
\end{center}
Let $C$ be a coloring of $\mathcal{A}^+$ such that $C(u)=C_{NF}(u)$ whenever $u$ is not a factor of $x$. Then the set 
\begin{center}
	$\left( u_{n_1}u_{n_2}\ldots u_{n_k} \right)_{k\geq 1, n_1<n_2<\ldots < n_k}$
\end{center}
is not $C$-monochromatic.

\end{prop}

\begin{proof}
Assume by contradiction that the set $(u_A)_{\mathcal{P}(\mathbb{N}^*)}$ is monochromatic. Take $\nu_1\geq 0$ and $A\geq \nu_1$ such that $u_A$ is not a factor of $x$. Take $\nu_2 > \max A $ and $B \geq \nu_2$ such that $u_B$ is not a factor of $x$, and set $z=u_Au_B$. The word $z$ belongs to the monochromatic set, is a non-factor of $x$ and if $z=v_1v_2$ where $v_1$ and $v_2$ are factors of $x$, then we have : $u_A$ is a factor of $v_1$ or $u_B$ is a factor of $v_2$, and at least one of $u_A$ or $u_B$ is a factor of $x$, which is impossible. This shows that any expression of $z$ as the concatenation of factors of $x$ must contain at least tree terms, so that $z$ is \textbf{red} and the factorisation is \textbf{red}.

Now let $i,j\in\mathbb{N}^*$ be such that $i+2\leq j$. If $u_iu_j$ is not a factor of $x$, then it must be $\textbf{blue}$, but this is a contradiction with the monochromatic assumption. Hence $u_iu_j$ is a factor of $x$, and actually the same proof shows that $u_A$ is a factor of $x$ whenever $A$ is the union of two intervals. Now if $A=I_1\cup I_2\cup I_3$ is the union of tree intervals, then since $u_{I_1\cup I_2}$ is a factor of $x$, $u_A$ is either \textbf{blue} or a factor of $x$. By the monochromatic assumption, $u_A$ is a factor of $x$. Proceeding by induction, we see that $u_A$ is a factor of $x$ for all finite subset $A$ of $\mathbb{N}$, but this is a contradiction with our first hypothesis.

\end{proof}

\begin{defi}
Let $C$ be the coloring defined on the set of factors $u$ of $x$ by :
\begin{itemize}
	\item $C(u)=\textbf{red}$ if and only if there exists a decomposition $u=v_1v_2$ with $A(v_1)=A(u)$ and $B(v_2)=B(u)$
	\item $C(u)=\textbf{blue}$ otherwise.
\end{itemize}

\end{defi}

\begin{prop}
If a suffix $y$ of $x$ has a super-monochromatic factorisation, then this suffix admits a super-monochromatic factorisation $y=u_1u_2u_3\ldots$ with $\forall n \geq 1$, $u_1u_2\ldots u_n$ is a suffix of $u_{n+1}$.
\end{prop}

\begin{proof}
Write $y=T^k(x)=u_1u_2u_3\ldots$ a super-monochromatic factorisation with the assumption that $A(u_1)=k$ and $B(u_i)=A(u_{i+1})$ for all $i\geq 1$. Assume also that $\forall n\geq 1$, $|u_{n+1}|\geq |u_1u_2\ldots u_n|$. We see in these conditions that $u_1u_2$ is \textbf{red}, so that the factorisation is \textbf{red}.

Let $n\geq 1$, and consider the \textbf{red} factor $u_1u_2\ldots u_n u_{n+2}$ of $x$. Write $u_1u_2\ldots u_n u_{n+2}=v_1v_2$ with $A(v_1)=A(u_1u_2\ldots u_n u_{n+2})$ and $B(v_2)=B(u_1u_2\ldots u_n u_{n+2})$. If $v_1$ is a prefix of $u_1u_2\ldots u_n$, then $A(v_1)\leq A(u_1u_2\ldots u_n)=A(u_1)$. Also in this case, $u_{n+2}$ is a suffix of $v_2$ so that $B(v_2)\geq B(u_{n+2})$. This implies, with $z=u_1u_2\ldots u_n u_{n+2}$,
\begin{center}
	$|z|= B(z)-A(z) \geq B(u_{n+2}) - A(u_1) = |u_1u_2\ldots u_nu_{n+1}u_{n+2}|=|z|+|u_{n+1}|> |z|$
\end{center}
wich is a contradiction.

So $v_2$ must be a suffix of $u_{n+2}$. So we have
\begin{center}
	$B(v_2)\leq B(u_{n+2}) \leq B(u_1u_2\ldots u_nu_{n+2})=B(v_2)$
\end{center}
So that $B(u_1u_2\ldots u_nu_{n+2})=B(u_{n+2})$. But this implies that $u_1u_2\ldots u_n$ is a suffix of $u_{n+1}$.
\end{proof}

\section{The example of the Zimin word}

In this section we study the Zimin word, which is an infinite word over an infinite alphabet. No knowledge of this word is required. We build a coloring answering to the conjecture that has 2 colors. In the end of this section we will present how to adapt these results to the doubling-period word, wich is a word over the alphabet $\{0,1\}$. When the author writes these lines, these are the only fully complete examples with 2 colors available.

\begin{defi}
Define the Zimin word $Z$ as the infinite word over the infinite alphabet $\mathcal{A}_x=\{x_1, x_2, x_3, \ldots \}$ by one of the tree equivalent definitions :
\begin{enumerate}
	\item $Z=\lim Z_n$ where $Z_1=x_1$, and $Z_{n+1}=Z_nx_{n+1}Z_n$ for $n\geq 1$
	\item $Z=\displaystyle \prod_{n\geq 1}x_{val_2(n)+1}$ where $val_2$ is the $2$-adic valuation.
	\item $Z$ is the fixed point of the morphism $\varphi : x_i \mapsto x_1x_{i+1}$ $(i\geq 1)$
\end{enumerate}
so that
\begin{center}
	$Z=x_1x_2x_1x_3x_1x_2x_1x_4x_1x_2x_1x_3x_1x_2x_1x_5x_1x_2x_1x_3\ldots$
\end{center}
\end{defi}

We leave the proof of the equivalence between these definitions to the reader, and will use the first one for proofs.

For a factor $u$ of $Z$, set
\begin{center}
	$k(u)=\max\{k\geq 0 \  |\ x_k \text{ appears in }u \}$,
\end{center}
and notice that for $k\geq 1$, $k=k(u)$ if and only if $u$ is a factor of $Z_k$ but is not a factor of $Z_{k-1}$. Since the letter $x_k$ appears only once in $Z_k$, we see that the letter $x_{k(u)}$ appears only once in the factor $u$. More generaly, between two occurences of a letter $x_k$, there must be a letter $x_l$ with $l>k$.

\begin{defi}
Define the two sequences of words $(u_n)_{n\geq 1}$ and $(v_n)_{n\geq 1}$ by $u_1=v_1=x_1$ and for $n\geq 1$ :
\begin{center}
	$u_{n+1}=x_{n+1}u_1u_2\ldots u_n$ \quad and \quad $v_{n+1}=v_nv_{n-1}\ldots v_1x_{n+1}$.
\end{center}
\end{defi}

For $A,B \subset \mathbb{N}$ two finite sets, let
\begin{center}
	$\displaystyle u_A=\prod_{i\in A}^{\rightarrow}u_i=u_{i_1}u_{i_2}\cdots u_{i_k}$ \ where \ $A=\{i_1<\ldots<i_k\}$
\end{center}
and
\begin{center}
	$\displaystyle v_B=\prod_{j\in B}^{\leftarrow}v_j=v_{j_k}v_{j_{k-1}}\cdots v_{j_1}$ \ where \ $B=\{j_1<\ldots<j_k\}$.
\end{center}

\begin{prop}
Let $n\geq 1$. Then
  \begin{enumerate}
	\item The proper suffixes of $u_n$ are exactly the words $u_A$ where $A\subset [1,n[$.
	\item For all $A\subset [1,n[$, $Z_{n-1}=v_{[1,n[\backslash A}u_A$ \ and \ $u_n=x_nv_{[1,n[\backslash A}u_A$
\end{enumerate}
\end{prop}

\begin{proof}
From the relation $|Z_{n+1}|=1+2|Z_n|$ and $|Z_1|=1$ one derive $|Z_n|=2^n-1$. Similarly we get $|u_n|=2^{n-1}$. It is clear that each $u_A$ for $A\subset [1,n[$ are proper suffixes of $u_n$. Since we have $|u_A|=\sum_{i\in A} 2^{i-1}$, we see that $u_A$ is characterized by its length, and that every length is obtained that way, proving the statement.

For the second relation, notice that $u_A$ is the reversal of $v_A$, and use the fact that $Z_n$ is a palindrome and $u_n$ is the suffix of length $2^{n-1}$ of $Z_n$. 
\end{proof}

\begin{prop}
Every factor $u$ of $Z$ writes uniquely in the form :
\begin{center}
	$u=u_Ax_{k(u)}v_B$
\end{center}
with $A,B\subset [1,k(u)[$.
\end{prop}

\begin{proof}
Unicity is clear by consideration of the lengths. For existence, write $Z_{k(u)}=\lambda u \rho$ and use the previous proposition.
\end{proof}

\begin{prop}
Let $u,v$ be two factors of $Z$ such that $k(u)<k(v)$, and write :
\begin{center}
	$u=u_{A_1}x_{k(u)}v_{B_1}$ \quad and \quad $v=u_{A_2}x_{k(v)}v_{B_2}$
\end{center}
with $A_1,B_1 \subset [1,k(u)[$ and $A_2,B_2 \subset [1,k(v)[$. Then :
\begin{enumerate}
	\item $uv$ is a factor of $Z$ if and only if $k(u)\notin A_2$ and $B_1=[1,k(u)[\backslash ( A_2\cap [1,k(u)[ )$.
	\item if $uv$ is a factor of $Z$ then
	\begin{center}
		$uv=u_{A_1\cup \{k(u)\}\cup (A_2\cap ]k(u),k(v)[)}x_{k(v)}v_{B_2}$
	\end{center}
	\item $u$ is a suffix of $v$ if and only if $k(u)\in B_2$ and $B_2\cap [1,k(u)[=B_1$
\end{enumerate}
\end{prop}

\begin{proof}
Assume that $uv$ is a factor of $Z$. Then we must have $k(u)\notin A_2$ for otherwise the factor $uv$ of $Z$ would contain two occurences of $x_{k(u)}$ and no letter $x_l$ with $l>k(u)$ between them. Now the recursive definition of $Z$ shows that each occurrence of a letter $x_k$ is followed by the word $Z_{k-1}$. This shows that $x_{k(u)}v_{B_1}u_{A_2\cap [1,k(u)[}=u_{\{k(u)\}}$ and this implies that $B_1=[1,k(u)[\backslash ( A_2\cap [1,k(u)[ )$. Conversely, if $B_1=[1,k(u)[\backslash ( A_2\cap [1,k(u)[ )$, then $x_{k(u)}v_{B_1}u_{A_2\cap [1,k(u)[}=u_{\{k(u)\}}$ and we have :
\begin{center}
	$uv=u_{A_1}u_{\{k(u)\}}u_{A_2 \cap ]k(u),k(v)[}x_{k(v)}v_{B_2}$
\end{center}
showing at once that $uv$ is a factor of $Z$ and the desired formula.

Since $k(u)<k(v)$, $u$ is a suffix of $v$ if and only if $u$ is a suffix of $v_{B_2}$. It is clear that $u$ is a suffix of $v_{k(u)}v_{B_1}$ so that it is enough to prove the property assuming $u=v_{k(u)}v_{B_1}$. Since $v_n$ ends with the letter $x_n$, $v_{k(u)}v_{B_1}$ is a suffix of $v_{B_2}$ if and only if it is a suffix of $v_{B_2\cap [1,k(u)]}$ and this shows that $k(u)\in B_2\cap [1,k(u)]$. A similar induction shows that $B_1\subset B_2\cap [1,k(u)[$. If $k\in B_2\cap [1,k(u)[$ and $k\notin B_1$, then in $v$ the letter $x_k$ has two occurrences, without a letter $x_l$ with $l>k$ between them, and this is a contradiction.
\end{proof}

\begin{coro}
Let $u,v$ and $w$ be tree factors of $Z$ with $k(u)<k(v)<k(w)$. Then :
\begin{enumerate}
	\item if $uv$ and $vw$ are factors of $Z$, then $uvw$ is a factor of $Z$.
	\item if $uw$ and $vw$ are factors of $Z$, then $u$ is a suffix of $v$.
\end{enumerate}
\end{coro}

\begin{proof}
Write $u=u_{A_1}x_{k(u)}v_{B_1}$, $v=u_{A_2}x_{k(v)}v_{B_2}$ and $w=u_{A_3}x_{k(w)}v_{B_3}$.

From the formula obtained for $uv$, and the fact that $vw$ being a factor only relies on relations between $k(v)$, $B_2$ and $A_3$, we see that, assuming that $uv$ is a factor of $Z$, the word $vw$ is a factor of $Z$ if and only if $uvw$ is a factor of $Z$. Moreover we have :
\begin{center}
	$uvw=u_{A_1\cup \{k(u)\}\cup (A_2\cap ]k(u),k(v)[)\cup \{k(v)\}\cup (A_3\cap ]k(v),k(w)[)}x_{k(w)}v_{B_3}$.
\end{center}

The second statement is similar : $uw$ is a factor of $Z$ implies $k(u)\notin A_3$, and since $B_2=[1,k(v)[\backslash (A_3 \cap [1,k(v)[)$ so that $k(u)\in B_2$. Moreover,
\begin{center}
	$B_2\cap [1,k(u)[=[1,k(u)[\backslash A_3\cap [1,k(u)[= B_1$
\end{center}
showing that $u$ is a suffix of $v$.
\end{proof}

Recall that if $A, B$ are finite subsets of $\mathbb{N}$, we write $A<B$ if $\max A < \min B$. We say that $A\subset\mathbb{N}$ is an interval if there are $k\leq l\in\mathbb{N}$ such that $A=[k,l]$.

\begin{lem}
Let $(A_n)_{n\geq 1}$ be a sequence of finite subsets of $\mathbb{N}^*$ with $A_n<A_{n+1}$ for all $n\geq 1$. Then the infinite word
\begin{center}
	$Y=\displaystyle \prod_{n\geq 0}u_{A_n}$
\end{center}
is a suffix of $Z$ if and only if $\exists N\in\mathbb{N}, \forall n\geq N$, $A_n$ is an interval and $\min A_{n+1} = 1+\max A_n$.
\end{lem}

\begin{proof}
Notice first that 
\begin{center}
	$Z=\displaystyle \prod_{i\geq 1} u_i$.
\end{center}
If $M\subset \mathbb{N}^*$ is an infinite set, then the infinite word
\begin{center}
	$Y_M=\displaystyle \prod_{m\in M} u_m$
\end{center}
belongs to $\Omega(Z)$. Conversely, for $Y\in \Omega(Z)$, there exists $M\subset \mathbb{N}^*$ such that $Y=Y_M$. To see this, build $M=\{m_1<m_2<\ldots\}$ as follows. Let $x_{m_1}$ be the first letter of $Y$. By the recursive definition of $Z$ we see that $u_{m_1}$ is a prefix of $Y$. Let $Y_1=Y$ and $Y_2$ be the suffix of $Y$ starting where the prefix $u_{m_1}$ of $Y_0$ ends. If $x_{m_2}$ is the first letter of $Y_1$, we must have $m_2>m_1$ and $u_{m_2}$ is a prefix of $Y_2$, so that $u_{m_1}u_{m_2}$ is a prefix of $Y$. If $u_{m_1}u_{m_2}\ldots u_{m_k}$ is a prefix of $Y$ with $m_1<m_2<\ldots<m_k$ and $Y_{k+1}$ is defined as the suffix of $Y$ where $u_{m_1}u_{m_2}\ldots u_{m_k}$ ends, then let $m_{k+1}$ be such that $Y_{k+1}$ starts with the letter $x_{m_{k+1}}$. We must have $m_{k+1}>m_k$, and $u_{m_{k+1}}$ is a prefix of $Y_{k+1}$, so that $u_{m_1}u_{m_2}\ldots u_{m_k}u_{m_{k+1}}$ is a prefix of $Y$. The infinite set $M=\{m_1<m_2<m_3<\ldots\}$ defined this way is such that $Y=Y_M$. Moreover, this construction shows that $M$ is uniquely determined by $Y$.

Since every suffix of $Z$ is of the form $Y_M$ for some infinite $M\subset \mathbb{N}^*$ such that $\exists a\in \mathbb{N}^*$ with $[a,+\infty[\subset M \subset \mathbb{N}^*$, we see that if
\begin{center}
	$Y=\displaystyle \prod_{n\geq 0}u_{A_n}$
\end{center}
with $A_n<A_{n+1}$ for all $n\geq 1$, we must have
\begin{center}
	$\exists a\in \mathbb{N}^*$, \ such that \ $[a,+\infty[\subset \displaystyle \bigcup_{n\geq 1} A_n \subset \mathbb{N}^*$
\end{center}
proving the lemma.

\end{proof}

\begin{defi}
Let $u=u_Ax_{k(u)}v_B$ be a factor of $Z$, and set $\eta(u)=\max([1,k(u)[\backslash A)$, with the convention that $\eta(u)=0$ if $A=[1,k(u)[$. Let $C$ be the coloring defined by :
\begin{itemize}
	\item $C(u)=\textbf{red}$ if $A\cap [1,\eta(u)[=[1,\eta(u)[\backslash(B\cap [1,\eta(u)[)$,
	\item $C(u)=\textbf{blue}$ otherwise.
\end{itemize}
If $u$ is not a factor of $Z$, then we set $C(u)=C_{NF}(u)$ where the coloring $C_{NF}$ is defined in section $2$ with $x=Z$.
\end{defi}

\begin{prop}
Let $A\subset \mathbb{N}$ be a finite subset. Then $u_A$ is \textbf{red} if and only if $A$ is an interval.
\end{prop}

\begin{proof}
Let $k=k(u_A)=\max A$, and set $A_0=A\backslash \{k\}$. We have
\begin{center}
	$u_A=u_{A_0}x_{k}v_{[1,k[}$
\end{center}
so that $u_A$ is \textbf{red} if and only if $A_0\cap[1,\eta(u_A)[=\emptyset$. But $\eta(u_A)=\max([1,k[\backslash A_0)$, so that $A_0\cap[1,\eta(u_A)[=\emptyset$ if and only if $A_0=]\eta(u_A),k[$, meaning that $A=]\eta(u_A),k]$ is an interval.
\end{proof}

\begin{theo}
The Zimin word $Z$ admits no suffix $Y$ having a super-monochromatic factorisation.
\end{theo}

\begin{proof}

Assume by contradiction that there exists a suffix $Y$ of $Z$ having a super-monochromatic factorisation
\begin{center}
	$Y=\displaystyle \prod_{n\geq 1}w_n$.
\end{center}
We assume that $k(w_n)\leq 2+k(w_{n+1})$ for all $n\geq 1$.

By proposition 5 we may assume that $\prod_{i\in A}w_i$ is a factor of $Z$ for all $A\subset \mathbb{N}^*$ finite. Let $n\geq 1$, and consider the factor $w_1w_2\ldots w_nw_{n+2}$ of $Z$. Since $w_{n+1}w_{n+2}$ is a factor of $Z$, we have by proposition 9 that $w_1w_2\ldots w_n$ is a suffix of $w_{n+1}$.

Write
\begin{center}
	$w_n=u_{A_n}x_{k(w_n)}v_{B_n}$
\end{center}
for all $n\geq 1$. We have
\begin{center}
	$Y=\displaystyle \prod_{n\geq1}u_{A_n}x_{k(w_n)}v_{B_n}=u_{A_1}\prod_{n\geq1}x_{k(w_n)}v_{B_n}u_{A_{n+1}}$
	\bigskip
	$\displaystyle Y=u_{A_1}\prod_{n\geq1}u_{\{k(w_n)\}\cup (A_{n+1}\cap]k(w_n),k(w_{n+1})[)}$.
\end{center}
By the Lemma, the sets $\{k(w_n)\}\cup (A_{n+1}\cap]k(w_n),k(w_{n+1})[)$ become intervals for large $n$. So that there exist $N\geq 1$ such that for all $n\geq N$,
\begin{center}
	$\{k(w_n)\}\cup (A_{n+1}\cap]k(w_n),k(w_{n+1})[)=[k(w_n),k(w_{n+1})[$
\end{center}
and $\emptyset \neq ]k(w_n),k(w_{n+1})[ \subset A_{n+1}$.
But the fact that $w_nw_{n+1}$ is a factor of $Z$ implies that $k(w_n)\notin A_{n+1}$, and all this shows that
\begin{center}
	$\eta(w_{n+1})=k(w_n)$.
\end{center} 
Now, since $w_n$ is a suffix of $w_{n+1}$, we have $B_{n+1}\cap[1,k(w_n)[=B_n$, and since $w_nw_{n+1}$ is a factor of $Z$, we have
\begin{center}
$A_{n+1}\cap [1,k(w_n)[ = [1,k(w_n)[ \backslash B_n = [1,k(w_n)[ \backslash (B_{n+1}\cap[1,k(w_n)[)$
\end{center}
and with the fact that $\eta(w_{n+1})=k(w_n)$ we see that $w_{n+1}$ is \textbf{red}. Thus the factorisation is super-monochromatic with respect to the color \textbf{red}.

This implies that $w_nw_{n+2}$ is \textbf{red}. We have
\begin{center}
	$w_nw_{n+2}= u_{A_n\cup\{k(w_n)\}\cup(A_{n+2}\cap]k(w_n),k(w_{n+2})[)}x_{k(w_{n+2})}v_{B_{n+2}}$.
\end{center}
But $]k(w_{n+1}),k(w_{n+2})[ \subset A_{n+2}$ and $k(w_{n+1})\notin A_{n+2}$, so that $\eta(w_nw_{n+2})=k(w_{n+1})$. By the red condition, we have
\begin{align*}
	& \ A_n\cup\{k(w_n)\}\cup(A_{n+2}\cap]k(w_n),k(w_{n+1})[) \\
	=& \ [1,k(w_{n+1})[\backslash (B_{n+2}\cap[1,k(w_{n+1})[) \\
	=& \ [1,k(w_{n+1})[\backslash B_{n+1} \\
	=& \ A_{n+2}\cap [1,k(w_{n+1})[
\end{align*}
and we see that $k(w_n)\in A_{n+2}$. But since $w_nw_{n+2}$ is a factor of $Z$, $k(w_n)\notin A_{n+2}$, which is a contradiction.

\end{proof}

We end this section by producing a coloring for the doubling-period word $D$, with two colors, such that $D$ admits no suffix having a super-monochromatic factorisation. We mention \cite{damanik} for a computation of squares in the doubling-period word.

\begin{defi}
The doubling-period word $D$ is the infinite word over the alphabet $\mathcal{A}=\{0,1\}$ defined as $D=\psi(Z)$ where $\psi$ is the morphism $\mathcal{A}_x^+ \rightarrow \mathcal{A}^+$ defined by $\psi(x_n)=0$ if $n\geq 1$ is odd, and $\psi(x_n)=1$ is $n$ is even. We have :
\begin{center}
	$D=01000101010001000100010\ldots$
\end{center}
\end{defi}

For a factor $u$ of $D$, define the sets $\psi^{-1}(u)=\{V \ \text{ factor of } Z | \  \psi(V)=u\}$ and $\psi_k^{-1}(u)=\{V \in \psi^{-1}(u) \ |\  k(V)=k \}$.

Let $W(u)$ be the element $V$ of $\psi^{-1}(u)$ such that $A_Z(V)$ is minimal. By minimality and existence, we have $A_Z(W(u))=A(u)$.

Let $C_Z$ be the coloring answering the conjecture for the Zimin word. We define the coloring $C$ on the set of factors of $D$ by
\begin{center}
	$C(u)=C_Z(W(u))$
\end{center}

\begin{theo}
No suffix of $D$ admits a super-monochromatic factorisation for the coloring defined above.
\end{theo}

\begin{proof}
Let $k\geq 0$ and
\begin{center}
	$y=T^k(D)=u_1u_2u_3\ldots$
\end{center}
be a suffix of $D$ and a super-monochromatic factorisation such that $k=A_D(u_1)=A_Z(W(u_1))$ and $B(u_i)=A(u_{i+1})$ for all $i\geq 1$.

Let $u$ be a factor of $D$ such that $A(u)=k$. Write $W(u)=u_Ax_{k(W(u))}v_B$ as a factor of $Z$. Assume that the three letters $x_{k(W(u))-2}$, $x_{k(W(u))-1}$ and $x_{k(W(u))}$ all appear in $W(u)$, with this order of apparition. Let $V_1$ and $V_2$ be two factors of $Z$ such that $\psi(V_1)=\psi(V_2)$, then in these words the occurrences of the letter $x_1$ coincide. Indeed, $|W(u)|\geq 3$ and it is easily seen that this inequality is optimal in order to find the possible positions of $x_1$ in $V_1$ and $V_2$. We can then erase the letters $x_1$ from $V_1$ and $V_2$ an proceed by induction to see that each occurrences of letters $x_l$ with $l\leq k(W(u))-2$ is uniquely determined in $V_1$ and $V_2$. So that the words $V_1$ and $V_2$ are equal to $W(u)$ up to the occurrences of the letters $x_{k(W(u))-1}$ and $x_{k(W(u))}$. But these two letters have different images through $\psi$, and since between two letters $x_l$ and $x_k$ all letters $x_j$ with $j< l$ occur, we see that $V_1$ and $W(u)$ are equal up to the letter $x_{k(W(u)}$. This means that if
\begin{center}
	$W(u)=u_Ax_{k(W(u))}v_B$
\end{center}
then $V=u_Ax_{k(W(u))+2m}v_B$ for some $m\geq 0$.

Now consider the factor $W(u_iu_j)$ of $Z$ with $i\leq j-2$. Write $W(u_iu_j)=V_1V_2$ with $\psi(V_1)=u_i$ and $\psi(V_2)=u_j$. Assume that in $W(u_i)$the three letters $x_{k(W(u_i))-2}$, $x_{k(W(u_i))-1}$ and $x_{k(W(u_i))}$ appear with this order of apparition. Write
\begin{center}
	$V_1=u_{A_1}x_{k(W(u_i))+2m}v_{B_1}$ and $V_2=u_{A_2}x_{k(W(u_j))+2m'}v_{B_2}$
\end{center}
Since $V_1V_2$ is a factor of $Z$, between the two letters $x_{k(W(u_i))+2m}$ and $x_{k(W(u_j))+2m'}$ must appear every letter $x_j$ with $j\leq \min\{k(W(u_i))+2m , k(W(u_j))+2m'\}$, we see that we me must have $m=m'=0$. Showing that $V_1=W(u_i)$ and $V_2=W(u_j)$, so that $W(u_iu_j)=W(u_i)W(u_j)$.

This shows inductively that $W(u_{n_1}u_{n_2}\ldots u_{n_k})=W(u_{n_1})W(u_{n_2})\ldots W(u_{n_k})$ for all $k \geq 1, n_1<n_2<\ldots < n_k$. But this implies that $Z$ has a super-monochromatic factorisation for the coloring $C_Z$, leading to a contradiction.

\end{proof}

\section{Consecutive length}

Let $x$ be a non-ultimately periodic word. In this section, we introduce and study the consecutive length $L(u)$ of a factor $u$ of $x$.

Let $u$ be a factor of $x$. A decomposition $u=v_1v_2\ldots v_l$ with $l\geq 1$ terms is said to be consecutive if
\begin{center}
	$A(v_1)=A(u)$, $B(v_l)=B(u)$ \ and \ $\forall i=1\ldots l-1, \ B(v_i)=A(v_{i+1})$.
\end{center}

Define the consecutive length $L(u)$ of a factor $u$ of $x$ as :
\begin{center}
	$L(u)=\max\{ l \ | \ u \text{ admits a consecutive decomposition with } l \text{ terms} \}$
\end{center}

A factor $v$ of $x$ is said to be irreducible if $L(v)=1$. A consecutive decomposition is said to be irreducible if every of its terms is irreducible.

\begin{prop}
A consecutive decomposition of $u$ with $L(u)$ terms is irreducible.
\end{prop}

\begin{proof}
In such a decomposition $u=v_1v_2\ldots v_{L(u)}$, we must have $L(v_i)=1$ for all $i=1\ldots L(u)$ by maximality of the value of $L(u)$. Notice also that $L(v_i\ldots v_j)=j-i+1$ for all $i\leq j$.
\end{proof}

\begin{prop}
Let $u,v$ be two factors of $x$ with $B(u)=A(v)$. Then 
\begin{center}
	$L(u)+L(v) \leq L(uv) \leq L(u)+L(v)+1$.
\end{center}
\end{prop}

\begin{proof}
Let $u=u_1u_2\ldots u_{L(u)}$ and $v=v_1v_2\ldots v_{L(v)}$ be maximal decompositions of $u$ and $v$. Then $uv=u_1u_2\ldots u_{L(u)}v_1v_2\ldots v_{L(v)}$ is a consecutive decomposition, proving the first inequality.

On the other hand, let $uv=w_1w_2\ldots w_{L(uv)}$ be a maximal decomposition of $uv$. Let $i$ be such that $w_1w_2\ldots w_i$ is a prefix of $u$ and $w_{i+2}\ldots w_{L}$ is a suffix of $v$. Then we have $i\leq L(u)$ and $L-i-1\leq L(v)$ by definitions of $L(u)$ and $L(v)$, so that $L(uv)-1\leq L(u)+L(v)$ proving the second inequality.
\end{proof}

\begin{prop}
Let $k\geq 0$. Then for all $l\geq 1$, there exists a factor $u$ of $x$ with $A(u)=k$ and $L(u)=l$.
\end{prop}

\begin{proof}
We first show that the result is true for arbitrary large $l\geq 1$. For any factorisation $T^k(x)=u_1u_2u_3\ldots$, if the $(u_i)$'s are long enough then this factorisation is consecutive. This shows that $\lim_i(u_1u_2\ldots u_i) = +\infty$.

For the remaining $l\geq 1$, let $u$ be a factor of $x$ with $L(u)\geq l$ and $A(u)=k$. Let $u=v_1v_2\ldots v_{L(u)}$ be a maximal consecutive decomposition of $u$. Then $L(v_1v_2\ldots v_l)=l$, proving the statement.

\end{proof}

\section{Case of arbitrary large square free words}

In this section we use the consecutive length to provide a coloring answering the conjecture for infinite words not containing arbitrary large squares. We mention the construction in \cite{shallit} of a cube-free word over a 2-letter alphabet not containing arbitrary large squares. 

\bigskip

Let $u$ be a factor of $x$. Define the four sets :

\begin{itemize}
	\item $\lambda_+(u)=\{ v \text{ irreducible with }B(vu)=B(u) \ | \  B(v)=A(u) \}$
	\item $\lambda_-(u)=\{ v \text{ irreducible with }B(vu)=B(u) \ | \  B(v)<A(u) \}$

	\item  $\rho_+(u)=\{ v \text{ irreducible suffix of }u \ | \  B(v)=A(u) \}$
	\item  $\rho_-(u)=\{ v \text{ irreducible suffix of }u \ | \  B(v)<A(u) \}$
\end{itemize}

\bigskip

We have $\lambda_+(u)\neq \emptyset $ \ and \ $\rho_+(u)\neq \emptyset $, by use of the properties of the consecutive length.

\bigskip

We proceed now to the definition of the coloring. We use a third color, without mentioning it, to provide us the suffix hypothesis. Namely, for a factor $u$ of $x$, $L(u)\geq 2$ if and only if there exist two factors $v,w$ of $x$ with $u=vw$ and $A(u)=A(v)$ and $B(u)=B(w)$.

\bigskip

Let $x$ be a non-ultimately periodic word. Define the coloring $C$ for a factor $u$ of $x$ with $L(u)\geq 2$ as :
\begin{itemize}
	\item $C(u)=\textbf{red}$ \ if $\lambda_+(u)=\rho_+(u)$ and  $\lambda_-(u)=\rho_-(u)$,
	\item $C(u)=\textbf{blue}$ \ otherwise.
\end{itemize}

\begin{theo}
Assume that $x$ contains no arbitrary large squares. Then no suffix of $x$ admits a super-monochromatic factorisation for the 3-coloring defined above.

\end{theo}

\begin{proof}

Let $Y=T^k(x)=u_1u_2u_3\ldots$ be a suffix of $x$ and a super-monochromatic factorisation. We have, by the suffix hypothesis and the properties of the consecutive length :
\begin{center}
	$\forall n\geq 1$, $u_1u_2\ldots u_{n}$ is a suffix of $u_{n+1}$

	$\forall n\geq 1$, $B(u_n)=A(u_{n+1})$.
\end{center}

We show that $u_{n+1}$ is \textbf{red}. We show that $\lambda_+(u_{n+1}) \subset \rho_+(u_{n+1})$.

Let $v\in \lambda_+(u_{n+1})$ so that $L(v)=1$,  $B(vu_{n+1})= B(u_{n+1})$ and $B(v)=A(u_{n+1})$.

We have \begin{center}
 $B(vu_{n+1})= B(u_{n+1}) \Longrightarrow$ $\left\{\begin{matrix} & v \text{ is a suffix of } u_n \\  & \text{or } u_n \text{ is a suffix of }v\end{matrix}\right.$
\end{center}
Since $v$ is irreducible, we see that $v$ must be a suffix of $u_n$. The suffix property implies that  $v$ is a suffix of $u_{n+1}$. And since  $B(v)=A(u_{n+1})$, we have $v\in \rho_+(u_{n+1})$

By a similar proof We obtain the other inclusions : $\rho_+(u_{n+1}) \subset \lambda_+(u_{n+1})$, \  $\lambda_-(u_{n+1}) \subset \rho_-(u_{n+1})$, \ and $\rho_-(u_{n+1}) \subset \lambda_-(u_{n+1})$.

This shows that $u_{n+1}$ is \textbf{red}, so the factorisation is \textbf{red}.

 Let $v\in \lambda_+(u_nu_{n+2})$, so that $L(v)=1$, and $B(v)=A(u_nu_{n+2})$. By the monochromatic hypothesis, the word $u_nu_{n+2}$ is \textbf{red}.
	
We have, \quad with $x=\ldots u_nu_{n+1}u_{n+2}\ldots$, $v\in \lambda_+(u_nu_{n+2})$ so that $v\in \rho_+(u_nu_{n+2})$. Since $v\in \rho_+(u_nu_{n+2})$, we see that $v$ is a suffix of $u_nu_{n+2}$. But obviously $|v|\leq |u_{n+1}|$, so by the suffix property, $v$ is a suffix of $u_{n+1}$. Since $v$ is a suffix of $u_{n+1}$, we have $v\in \rho_-(u_{n+2})=\lambda_-(u_{n+2})$. But $A(v)> B(u_n)$ and the suffix property imply that $u_n$ is a suffix of $v$. Now $vu_n$ is a suffix of $u_{n+1}$, so that $u_nu_n$ is an arbitrary large square that is also a factor of $x$. Contradiction.

\end{proof}


\begin{thebibliography}{9}

\bibitem{shallit}
\textsc{N. Rampersad, J. O. Shallit, M. Wang} 
\textit{Avoiding large squares in infinite binary words}	arXiv:math/0306081 [math.CO] , 2003

\bibitem{damanik}
\textsc{D. Damanik} 
\textit{Local symmetries in the period-doubling sequence}	Discrete Applied Mathematics 100 (2000) 115-121

\bibitem{wojcik}
\textsc{C. Wojcik, L. Q. Zamboni} (2018). \textit{Monochromatic factorisation of words and periodicity}. Mathematika, 64(1), 115-123. doi:10.1112/S0025579317000377

\end{thebibliography}
\end{document}